\documentclass[11pt]{amsart}
\usepackage{amssymb}
\usepackage{amscd}
\usepackage{amsmath}
\usepackage{amsrefs}
\usepackage[utf8]{inputenc}
\usepackage{hyperref}
\usepackage{a4wide}
\theoremstyle{plain}

\newtheorem{thm}{Theorem}
\newtheorem{cor}[thm]{Corollary}
\newtheorem{lem}[thm]{Lemma}

\newtheorem{prop}[thm]{Proposition}

\newcommand{\ZZ}{\mathbb{Z}}

\newcommand{\CC}{\mathbb{C}}

\newcommand{\Ann}[2]{\mathrm{Ann}_{#1}({#2})}

\newcommand{\m}{\mathfrak{m}}

\newcommand{\MaxSpec}[1]{\mathrm{MaxSpec}(#1)}
\newcommand*{\longhookrightarrow}{\ensuremath{\lhook\joinrel\relbar\joinrel\rightarrow}}

\title{A note on a paper by  Cuadra, Etingof and Walton}

\author{Christian Lomp and Deividi Pansera}

\address{Department of Mathematics, Faculty of Science, University of Porto, Rua Campo Alegre 687, 4169-007 Porto, Portugal}

\thanks{We would like to thank Chelsea Walton and Pavel Etingof very much for many clarifications about their papers \cite{EtingofWalton, CuadraEtingofWalton}. In particular we would like to thank Chelsea Walton for her careful reading of our manuscript and her constructive comments.
Thanks go also to the very detailed and constructive comments of the referee. Further thanks go to Ken Brown, Paula Carvalho, Jörg Feldvoss, André Leroy and Jurek Matczuk for helpful discussions. Both authors were partially supported by CMUP (UID/MAT/00144/2013), which is funded by FCT (Portugal) with national (MEC) and European structural funds through the programs FEDER, under the partnership agreement PT2020. The second author was also supported by CAPES, Coordination of Superior Level Staff Improvement - Brazil.}
\begin{document}

\begin{abstract}
We analyse the proof of the main result of a paper by Cuadra, Etingof and Walton, which says that any action of a semisimple Hopf algebra $H$ on the $n$th Weyl algebra $A=A_n(K)$ over a field $K$ of characteristic $0$ factors through a group algebra. We verify that their methods can be used to show that any action of a semisimple Hopf algebra $H$ on an iterated Ore extension of derivation type $A=K[x_1;d_1][x_2;d_2][\cdots][x_n;d_n]$ in characteristic zero factors through a group algebra.
\end{abstract}

\maketitle

The purpose of this note is to analyse the main result of the paper \cite{CuadraEtingofWalton} which says that any action of a semisimple Hopf algebra $H$ on the $n$th Weyl algebra $A=A_n(K)$ over a field $K$ of characteristic $0$ factors through a group algebra. The central idea is to pass from algebras in characteristic $0$ to algebras in positive characteristic by using the subring $R$ of $K$, generated by all structure constants of $H$ and the action on $A$ and by passing to a finite field $R/\m$. 

It has already been  outlined in \cite[p.2]{CuadraEtingofWalton} that these methods could be used to establish more general results on semisimple Hopf actions on quantized algebras and that the authors of \cite{CuadraEtingofWalton} will do so in their future work. In particular it  has been announced in \cite[p.2]{CuadraEtingofWalton} that their methods will apply to actions on module algebras A such that the resulting algebra $A_p$ when passing to a field of  characteristic $p$, for large $p$,  is PI and their PI-degree  is a power of $p$. Such algebras include universal enveloping algebras of finite dimensional Lie algebras and algebras of differential operators of smooth irreducible affine varieties.

We will verify part of their outlined program in Theorem \ref{MainResultCEW} and will show that any action of a semisimple Hopf algebra $H$ on an enveloping algebra of a finite dimensional Lie algebra or on an iterated Ore extension of derivation type $A=K[x_1;d_1][x_2;d_2][\cdots][x_n;d_n]$ in characteristic zero factors through a group algebra. Apart from this we give an alternative proof for the reduction step in Proposition \ref{reduction2p}.

\section{Inner faithful action}

Let $H$ be a finite dimensional Hopf algebra over a field $K$ and let $A$ be a (left) $H$-module algebra.
One says that a Hopf algebra $H$ acts {\it inner faithfully} on a left $H$-module algebra $A$ if $I\cdot A \neq 0$ for any non-zero Hopf ideal $I$ of $H$. For any Hopf algebra $H$ acting and a left $H$-module algebra $A$ there exists a largest Hopf ideal $I$ of $H$ such that $I\cdot A=0$ and $H/I$ acts inner faithfully on $A$, i.e. any Hopf action factors through an inner faithful action. 
The next Lemma is essentially contained in the proof of \cite[Proposition 2.4]{CuadraEtingofWalton}:

\begin{lem}[Cuadra-Etingof-Walton]\label{innerfaithfully} A finite dimensional Hopf algebra $H$ acts inner faithfully on an algebra $A$ if and only if  $A^{\otimes n}$ is a faithful left $H$-module for some $n>0$.
\end{lem}

\begin{proof} The statement follows from the proof of  \cite[Proposition 2.4]{CuadraEtingofWalton}, which we will sketch here: Let  $\mathrm{Ann}_H(M)$ denote the annihilator of a left $H$-module $M$. For any $m>0$ set $K_m=\Ann{H}{A^{\otimes m}}$, where $A^{\otimes m}$ is a left $H$-module by the diagonal $H$-action.  Since $A^{\otimes m}$ embeds into $A^{\otimes n}$ as left $H$-module, whenever $m\leq n$,  we can conclude that $K_m\supseteq K_{n}$. Clearly the descending chain of ideals $K_m$ stabilises at some index $n$ as $H$ is finite dimensional. Hence $K_m=K_n=K_{2n}:=K$ for all $m\geq n$.  Thus the annihilator of $A^{\otimes n}\otimes A^{\otimes n}$ with the component wise action of $H\otimes H$ is equal to $H\otimes K + K\otimes H$. Since 
$$ 0 = K \cdot A^{\otimes 2n} = \Delta(K)\cdot (A^{\otimes n} \otimes A^{\otimes n}), $$ 
we get $\Delta(K)\subseteq H\otimes K + K\otimes H$. As $\epsilon(K)1_A = K\cdot 1_A = 0$, $K$ is a coideal. 
Thus $K$ is a bi-ideal and $(H/K)^*$ is a sub-bialgebra of the finite dimensional Hopf algebra $H^*$. By \cite[7.6.1]{Radford}, $(H/K)^*$ is a Hopf subalgebra and hence $K$ is a Hopf ideal of $H$. As $H$ acts inner faithfully on $A$, $K=0$, i.e. $A^{\otimes n}$ is a faithful $H$-module.

For the converse, note that if $I$ is a Hopf ideal that annihilates $A$, then it would also annihilate $A^{\otimes n}$. Hence $I=0$.
\end{proof}

\section{Reduction to a finitely generated $\ZZ$-algebra}\label{ReductionSection1}

Let $H$ be a semisimple Hopf algebra over a field $K$ of characteristic $0$. By a classical result of Larson and Radford, $H$ is also cosemisimple. Suppose that $H$ acts on a domain $A$ which is finitely presented as an $K$-algebra, i.e. 
$A\simeq K\langle x_1,\ldots, x_n\rangle/P,$ with $P$ a finitely generated completely prime ideal of the free algebra $K\langle x_1, \ldots, x_n \rangle$. 
The structure constants of $H$ are the constants that define the Hopf algebra structure of $H$ over $K$. First of all, since $H$ is semisimple over a field, $H$ is finite dimensional, say of dimension $d$ and hence has a $K$-basis $\{b_1, \ldots, b_d\}$. We might assume that $1_H$ is one of the basis vectors. The Hopf algebra structure of $H$ is determined by a set of constants $\mu_k^{i,j}, \eta_{i,j}^k, \nu_j^i$ and $\epsilon(b_i)$ such that for all $1\leq i,j,k\leq d$:
$$b_i \cdot b_j = \sum_{k=1}^d {\mu_k^{i,j}}\: b_k, \qquad \Delta(b_k)=\sum_{i,j=1}^d {\eta_{i,j}^k}\: b_i \otimes b_j, \qquad S(b_i)=\sum_{j=1}^d {\nu_{j}^{i}}\: b_j.$$
As $H$ is semisimple and cosemisimple, there exists a left integral $t=\sum_{i=1}^d \tau_i b_i$ in $H$ and a left integral $t^*=\sum_{i=1}^d \tau_i^* b_i^*$ in $H^*$ with $\epsilon(t)=1$ and $t^*(1)=1$ for some $\tau_i, \tau_i^* \in K$, where $\{b_1^*, \ldots, b_d^*\}$ denotes the dual basis of $H$. Moreover the action of $H$ on $A$ is determined by the images of the action of the basis elements $b_i$ on the algebra generators $\overline{x_j}=x_j+P$ of $A$, i.e. 
$$ b_i \cdot \overline{x_j} = f_{ij}(\overline{x_1},\ldots, \overline{x_n}),$$
where $f_{ij}$ are non-commutative polynomials in $K\langle x_1, \ldots, x_n\rangle$.
Since $P$ is finitely generated, there are non-commutative polynomials $p_1,\ldots, p_m$ such that $P=\langle p_1, \ldots, p_m\rangle$.
We consider now the subring $R$ of $K$ generated by all constants
\begin{equation}\label{structureconstants}{\mu_k^{i,j}}, \:  {\eta_{i,j}^k}, \: {\nu_{j}^{i}}, \: \epsilon(b_i), \: \tau_i, \:  \tau_i^*, \: \mbox{coefficients of }f_{ij}, \: \mbox{coefficients of } p_1, \ldots, p_m.\tag{$\star$} \end{equation}
We call $R$ the {\it ring of structure constants of the action of $H$ on $A$}.

Let $H_R = \bigoplus_{i=1}^d Rb_i$. Then $H_R$ is a $R$-Hopf algebra with structure constants (\ref{structureconstants}). Let $A_R = R\langle x_1, \ldots, x_n\rangle/ {P'}$, where $P'$ is the ideal of $R\langle x_1, \ldots, x_n\rangle$ generated by $p_1,\ldots, p_m$. The action of $H$ on $A$ yields now an action of $H_R$ on $A_R$, i.e. we have a ring homomorphism $H_R \rightarrow \mathrm{End}_R(A_R).$  Since $t \in H_R$ still satisfies $\epsilon(t)=1$, $H_R$ is separable over $R$. Analogously as  $t^* \in H_R^*$  satisfies $t^*(1)=1$, $H_R^*$ is separable over $R$ (see \cite{KadisonStolin,Lomp}). Note that $H\simeq H_R {\otimes_R} K$ and  $A\simeq A_R {\otimes_R} K$ as algebras. Hence $A_R$ is again a domain. Moreover, $A_R$ shares the same properties of $A$, i.e., $A_R$ is finitely generated and finitely presented. 

\section{$R$ is a Hilbert Ring}\label{ReductionSection2}

Since $K$ is a field of characteristic $0$, $R$ is an integral domain that contains the integers $\ZZ$. Let $a_1, \ldots, a_s$ be a set of generators of $R$ as $\ZZ$-algebra and consider the surjective ring homomorphism
$$\varphi: \ZZ[y_1, \ldots, y_s] \rightarrow R, \qquad y_i \mapsto a_i.$$
Recall that a {\it Hilbert ring} (or Jacobson ring) is a ring $R$ such that  any prime ideal is the intersection of maximal ideals (see \cite{Goldman, Krull}). As $\ZZ$ is a Hilbert ring,  and Hilbert rings are closed under forming polynomial rings and factor rings,  $R$ is a Hilbert ring as well. In particular the prime ideal $0$ is equal to the Jacobson radical $\mathrm{Jac}(R)$ of $R$. If $\m$ is any maximal ideal of $R$, then $\varphi^{-1}(\m)$ is a maximal ideal of $\ZZ[y_1, \ldots, y_s]$ and hence 
$\varphi^{-1}(\m)\cap \ZZ $ is a maximal ideal of $\ZZ$ by \cite[Theorem 5]{Goldman}, i.e. there exists a prime number $p$ such that $\varphi^{-1}(\m)\cap \ZZ  = p\ZZ$. In particular $R/\m$ has positive characteristic $p$. By Noether normalization, $R/\m$ is a finite field extension of the prime subfield $\ZZ/p\ZZ$, i.e. $R/\m$ is a finite field.

\begin{prop}\label{reduction2p}
 Let $R$ be an integral domain of characteristic $0$ with $\mathrm{Jac}(R)=0$ such that $\mathrm{char}(R/\m)>0$ for all maximal ideals $\m$ of $R$. Let $H$ be an $R$-algebra that is free of finite rank over $R$. For any faithful  left $H$-module  $M$ that is free as an $R$-module and number $q>1$, there exists a set of maximal ideals $Y$ such that 
 \begin{enumerate}
 \item  $M/\m M$ is a faithful $H/\m H$-module for any $\m \in Y$;
 \item $\mathrm{char}(R/\m)>q$ for any $\m \in Y$;
 \item  the canonical homomorphism of $R$-algebras $ \Psi: H_R \longrightarrow  \prod_{\m\in Y} H/\m H$
 is injective.
 \end{enumerate}
\end{prop}

Before we prove Proposition \ref{reduction2p}, we need two elementary Lemmas:

\begin{lem}\label{infinitemaximals}  Let $R$ be an integral domain of characteristic $0$ with $\mathrm{Jac}(R)=0$ such that $\mathrm{char}(R/\m)>0$ for all maximal ideals $\m$ of $R$. Then, given $0 \neq a\in R$ and an integer $q>1$, the following holds $$\bigcap_{\m \in X_{a,q}} \m = 0$$
for $X_{a,q} = \{ \m\in \MaxSpec{R} \mid \mathrm{char}(R/\m)>q \mbox{ and } a\not\in \m\}$.
\end{lem}
\begin{proof}
Set $X=\MaxSpec{R}$. 
For each $\m \in X$ let $p_{\m} = \mathrm{char}(R/\m)$. Set $B=\{\m \in X \mid a \not\in \m \text{ and } p_{\m} \leq q \}$. Now, if $B=\emptyset$, then $a \in \bigcap_{\m \in X\setminus X_{a,q}} \m $ and so, since by hypothesis $a\neq 0$ and $0=\mathrm{Jac}(R)=\left(\bigcap_{\m \in X\setminus X_{a,q}} \m\right) \cap \left(\bigcap_{\m \in X_{a,q}} \m\right)$, $R$ being a domain implies that $\bigcap_{\m \in X_{a,q}} \m = 0$  as we want. If $B \neq \emptyset$, then the set $\{p_\m \mid \m\in B\}$ is finite, say $\{p_{\m_1},\hdots ,p_{\m_n} \}$. We define $b = \prod_{i=1}^{n} p_{\m_i}$, which satisfies $b \in \bigcap_{\m \in B} \m.$
So, $ab\in \bigcap_{\m \in X\setminus X_{a,q}} \m$, because if $ \m \in X\setminus X_{a,q}$ and $a\not\in \m$ then  $p_{\m}\leq q$. Hence $p_{\m}$ divides $b$ and, as $p_{\m} \in \m$, we have $ab\in \m$. By hypothesis, $a\neq 0$ and $R$ being a domain implies $\bigcap_{\m \in X\setminus X_{a,q}} \m \neq 0$. However, $0=\mathrm{Jac}(R)=\left(\bigcap_{\m \in X\setminus X_{a,q}} \m\right) \cap \left(\bigcap_{\m \in X_{a,q}} \m\right)$ shows $\bigcap_{\m \in X_{a,q}} \m = 0$.
\end{proof}

\begin{lem}\label{VectorspaceLemma}
Let $V$ be a vector space over $K$ and let $h_1, \ldots, h_n$ be  endomorphisms of $V$.
Then $h_1, \ldots, h_n$ are linearly independent if and only if there exist  $v_1, \ldots, v_m \in V$ and  $f_{jk}\in V^*$ for $1\leq j \leq n$ and $1\leq k \leq m$ such that the matrix $$\left( \sum_{k=1}^m f_{jk}(h_i(v_k))\right)_{1\leq i,j \leq n}$$ has non-zero determinant.
\end{lem}
\begin{proof}
Suppose that $h_1, \ldots, h_n$ are linearly independent. Then there exists a finite dimensional subspace $U$ of $V$ of dimension $q\leq n$ such that $h_1, \ldots, h_n$ restricted to $U$ are linearly independent.
Let $\{v_1, \ldots, v_q\}$ be a basis for $U$. For each $1\leq i \leq n$ we define  $\varphi_i \in ((V^*)^q)^*$ as follows:
$$ \varphi_i(f) := \sum_{j=1}^q f_j(h_i(v_j)), \qquad \forall f=(f_1,\ldots, f_q)\in (V^*)^q.$$
Since $h_1, \ldots, h_n$ are linearly independent also  $\varphi_1, \ldots, \varphi_n$ are linearly independent.  
Let $W$ be the intersection of all kernels of $\varphi_i$ and consider the function
$$\Phi: (V^*)^q/W \longrightarrow K^n \qquad \mbox{ with } \qquad f+W \mapsto (\varphi_1(f), \cdots, \varphi_n(f)).$$
As $\Phi$ is injective, $\mathrm{dim}((V^*)^q/W)\leq n$. Moreover, as $\varphi_i$ are linearly independent in $((V^*)^q)^*$ they are also linearly independent in $((V^*)^q/W)^*$. Hence $\mathrm{dim}((V^*)^q/W)=\mathrm{dim}(((V^*)^q/W)^*)=n$. In particular $\Phi$ is an isomorphism and there are elements $f_1,\ldots, f_n \in (V^*)^q$ with 
$f_l=(f_{l1}, \ldots, f_{lq})$ for all $1\leq l \leq n$, such that the matrix
$$ \left( \varphi_i(f_l)\right)_{1\leq i,l\leq n} = \left(  \sum_{j=1}^q f_{lj}(h_i(v_j)) \right)_{1\leq i,l\leq n}$$
has non-zero determinant. The converse is clear.
\end{proof}

\begin{proof}[Proof of Proposition \ref{reduction2p}]
Let $R, H$ and $M$ as in the statement of the Proposition. The $H$-action of $H$ on $M$ is given by $n$ endomorphisms $h_i:M\rightarrow M$ for $i=1,\ldots, n$. Since $M$ is faithful, the elements $h_1, \ldots, h_n$ are independent over $R$ in the sense that if $\sum_{i=1}^n r_ih_i = 0$ for some $r_1, \ldots, r_n\in R$, then $r_1=\cdots = r_n=0$.  Let $F$ be the  field of fractions of $R$ and consider $M'=M\otimes_R F$ with its $F$-linearly independent endomorphisms $h_i' = h_i\otimes id_F: M' \rightarrow M'$. By Lemma \ref{VectorspaceLemma} there exist elements $v_k\in M'$ and linear functions $f_{jk}:M'\rightarrow K$ such that
$$0\neq d=\mathrm{det}\left( \sum_{k=1}^m f_{jk}(h_i'(v_k))\right)_{1\leq i,j \leq n}.$$
Let $C\in R$ be such that $w_k:=Cv_k \in M=M\otimes 1 $ for all $k$.
Let $\{b_\lambda \mid \lambda\in \Lambda\}$ be a basis for $M$ as an $R$-module, then there exists a finite subset $\Lambda' \subseteq \Lambda$ such that all elements $h_i(w_k)$ belong to the submodule spanned freely by $b_\lambda$ for $\lambda\in\Lambda'$. Let $D\in R$ be the common denominator of $f_{jk}(h_i(w_k))$ for all $i,j,k$ and define $R$-linear maps $g_{jk}:M\rightarrow R$ by $g_{jk}(b_\lambda)=Df_{jk}(b_\lambda)$ for $\lambda \in \Lambda'$ and $g_{jk}(b_\lambda)=0$ for $\lambda \in \Lambda\setminus \Lambda'$. Then 
$$a:=\mathrm{det}\left( \sum_{k=1}^m g_{jk}(h_i(w_k))\right)_{1\leq i,j \leq n} = C^n\mathrm{det}\left( \sum_{k=1}^m g_{jk}(h_i'(v_k))\right)_{1\leq i,j \leq n}=C^nD^nd \neq 0.$$
By Lemma \ref{infinitemaximals} for any $q>1$ the set $Y:=X_{a,q}$ satisfies $\bigcap_{\m \in Y} \m = 0$. 

\vspace{0.3cm}

\noindent $(1)$ Let  $\m \in Y$ and $\pi:R\rightarrow R/\m$ be the canonical projection. Consider $M\otimes_R R/\m = M/\m M$ and $H\otimes_R R/\m = H/\m H$, which acts on $M/\m M$ by the induced $R/\m$-linear maps $\overline{h_i}: M/\m M \rightarrow M/\m M$. Set $\overline{m}=m + \m M$ for all $m\in M$ and define $\overline{g_{jk}}:M/\m M \rightarrow R/\m$ by $\overline{g_{jk}}(\overline{m}) :=  \pi g_{jk}(m)$ for all $m\in M$. 
Then $\overline{a}=\mathrm{det}\left( \sum_{k=1}^m \overline{g_{jk}}(\overline{h_i}(\overline{w_k)})\right)_{1\leq i,j \leq n}$  is non-zero as $a\not\in \m$. By Lemma \ref{VectorspaceLemma}, $M/\m M$ is a faithful $H/\m H$-module.

\vspace{0.3cm}

\noindent $(2)$ By definition, for any $\m \in Y$, we have $\mathrm{char}(R/\m)>q$.

\vspace{0.3cm}

\noindent $(3)$ Since $\bigcap_{\m \in Y} \m = 0$,  $R$ is a subdirect product of the factor rings $R/\m$ for $\m \in Y$,  the canonical homomorphism $\pi_Y: R \hookrightarrow \prod_{\m\in Y} R/\m$ is injective. Tensoring with the free finite rank $R$-module $H$ yields an injective (ring) homomorphism
$$ 1\otimes \pi_Y: H_R \longhookrightarrow  \prod_{\m \in Y} H_R \otimes_R R/\m = \prod_{\m \in Y} H/\m H.$$
\end{proof}

\section{Reduction to Hopf algebras over finite fields}
Let $R$ be an integral domain in characteristic $0$ with zero Jacobson radical, such that $\mathrm{char}(R/\m)>0$ for all
 $\m \in \MaxSpec{R}$. Let $H$ be a separable and coseparable Hopf algebra over $R$ and free of finite rank as an $R$-module. For any maximal ideal $\m$ of $R$ the set $\m H$ is an ideal of $H$. Moreover, $H/\m H$ has an induced Hopf algebra structure over the finite field $F_\m=R/\m$. Denote by $\overline{x}$ the elements in $R/\m$ respectively $H/\m H$. The integral $t \in H$ with $\epsilon(t)=1$ yields an integral $\overline{t} \in H/\m H$ with
 $\epsilon(\overline{t}) = \overline{\epsilon(1)} = \overline{1}$. Hence $H/\m H$ is semisimple over $F_\m$. Analogously $(H/\m H)^*$ is semisimple. Given a left $H$-module algebra $A$ that is free as an $R$-module, we extend the $H$-action on $A$ to $A/\m A$ by   $\overline{h}\cdot \overline{a} = \overline{h\cdot  a}$, for all $a\in A, h\in H$. Hence $A/\m A$ is a left $H/\m H$-module algebra over a finite field $F_\m$. This reduction from a semisimple Hopf algebra action on a finitely presented algebra over a field of characteristic zero to an action of a semisimple and cosemisimple Hopf algebra over a finite field is the key step in \cite{CuadraEtingofWalton}. In some cases the algebras $A/\m A$ over $F_\m$ become finitely generated over their centre. Extending the Hopf action to the skew-field of fractions of $A/\m A$, when the degree of the skew-field of fractions and the $\mathrm{dim}(H)!$ are coprime, Cuadra, Etingof and Walton showed that $H/\m H$ has to be cocommutative, hence a group algebra. By Proposition \ref{reduction2p}(3) one concludes that $H$ has to be cocommutative, and hence a group algebra. Their key result of \cite{CuadraEtingofWalton} extending \cite{EtingofWalton} is the following:
 \begin{prop}[Cuadra-Etingof-Walton, {\cite[Proposition 3.3]{CuadraEtingofWalton}}]\label{prop33}
Let $F$ be an algebraically closed field, $H$ a semisimple, cosemisimple Hopf algebra over $F$ acting inner faithfully on a division algebra $D$ which is a finite module over its centre $Z$. If $[D:Z]$ and  $\mathrm{dim}(H)!$ are coprime,  then $H$ is a group algebra.
\end{prop}

As a consequence from this Proposition and the reduction process to fields of positive characteristic, as described in sections \ref{ReductionSection1} and \ref{ReductionSection2}, one deduces:

\begin{thm}[Cuadra-Etingof-Walton]\label{MainResultCEW}
Let $H$ be a semisimple Hopf algebra over a field $K$ of characteristic $0$ acting on a finitely presented algebra $A$ that is a Noetherian domain. Let $R$ be the ring of structure constants of $H$ and $A$ as defined by (\ref{structureconstants}) and let $H_R$ and $A_R$ be the corresponding $R$-algebras.
Suppose that there exists $q\geq 1$ such that for all maximal ideals $\m$ of $R$ with $\mathrm{char}(R/\m)>q$ one has:
\begin{itemize}
	\item the induced algebra $A_\m = A_R \otimes_R R/\m$ is a Noetherian domain; 
	\item the skew-field of fractions $D_\m$ of $A_\m$ is finite over its centre $Z_\m$;
	\item $[{D_\m}:{Z_\m}]$ is coprime with  $\mathrm{dim}(H)!$.
\end{itemize}
Then the action of $H$ on $A$ factors through a group algebra.
\end{thm}

\begin{proof} Suppose that $H$ acts inner faithfully on $A$. By Lemma \ref{innerfaithfully}, $H$ acts faithfully on $A^{\otimes_K n}$ for some $n$. Passing from $K$ to $R$, we also have that $H_R$ acts faithfully on $A_R^{\otimes_R n}$. By Proposition \ref{reduction2p} there exists a set $Y$ of maximal ideals of $R$ with $\mathrm{char}(R/\m)>q$ and 
$A_R^{\otimes_R n} \otimes_R R/\m = \left(A_R \otimes_R R/\m\right)^{\otimes_{R/\m} n}$ being a faithful 
$H_\m := H_R \otimes_R R/\m$-module for all $\m \in Y$. Again by Lemma \ref{innerfaithfully} $H_\m$ acts inner faithfully on $A_\m := A_R\otimes_R R/\m$. By assumption, the skew-field of fractions $D_\m$ of $A_\m$ is finite over its centre and its  dimension is coprime with  $\mathrm{dim}(H)!=\mathrm{dim}(H_\m)!$.  By \cite[Theorem 2.2]{SkryabinOystaeyen} the action of $H_\m$ on $A_\m$ extends to an action on $D_\m$, which must be also inner faithful. The same is true if we pass to the algebraic closure $\overline{R/\m}$ of $R/\m$ and tensor up $H_\m, A_\m$ and $D_\m$. By Proposition \ref{prop33}  $H_\m\otimes_{R/\m} \overline{R/\m}$ is cocommutative and hence $H_\m$ is cocommutative. By Proposition \ref{reduction2p}(3), the canonical $R$-algebra homomorphism $H_R \longhookrightarrow  \prod_{\m \in Y} H_\m = H_R \otimes_R \prod_{\m \in Y} R/\m$ is injective. Since all $H_\m$ are cocommutative, also $H_R$ is cocommutative, and therefore H is as well.  
\end{proof}

\section{Enveloping algebras of finite dimensional Lie algebras}

The main purpose in \cite{CuadraEtingofWalton}  was to show that any semisimple Hopf  action on a Weyl algebra factors through a group algebra. However, Cuadra, Etingof and Walton's method can also be used to show that actions of semisimple Hopf algebras on enveloping algebras of finite dimensional Lie algebras or on iterated differential operator rings over a field $K$ of  characteristic $0$, factor through group algebras.
For modular Lie algebras one has the following result from Farnsteiner and Strade's book {\cite[Chapter 6, Theorem 6.3(1)]{StradeFarnsteiner}}:
\newcommand{\g}{\mathfrak{g}}

\begin{thm}[Farnsteiner-Strade]\label{FarnsteinerStrade}
Let $U(\g)$ be the enveloping algebra of a  finite dimensional Lie algebra $\g$ over a field of characteristic $p$. Then the dimension of  $\mathrm{Frac}(U(\g))$ over its centre is a power of $p$, 
\end{thm}

This Theorem readily asserts the following:

\begin{cor}\label{Enveloping} Any action of a semisimple  Hopf algebra over a field $F$ of characteristic zero  
on the enveloping algebra  of a  finite dimensional Lie algebra factors through a group algebra.
\end{cor}

\begin{proof}
Let $H$ be a semisimple Hopf algebra over a field $K$ of characteristic $0$.
Let $\mathfrak{g}$ be a finite dimensional Lie algebra over $K$ and $A:=U(\mathfrak{g})$ its enveloping algebra.
Suppose that $H$ acts on $A$ and denote by $R$ the ring of structure constants of $H$ and $R$ as in (\ref{structureconstants}). Using the structure constants of $\g$, respectively $H$, define the Lie algebra $\mathfrak{g}_R$ over $R$ of finite rank and the Hopf algebra $H_R$ over $R$. Moreover, $A_R$ is $U(\mathfrak{g}_R)$, the enveloping algebra of $\mathfrak{g}_R$ over the integral domain $R$.
Let $\m$ be any maximal ideal of $R$, set $F=R/\m$ and $p=\mathrm{char}(F)$. Then
$A_\m:={A_R}   {\otimes_R}  {F} = {U({\mathfrak{g}_R})} {\otimes_R} {F} \simeq U( {{\mathfrak{g}_R} {\otimes_R} {F}})$
is the enveloping algebra of the finite dimensional Lie algebra $\mathfrak{g}_\m:={{\mathfrak{g}_R} {\otimes_R} {F}}$ over the finite field $F$. By Theorem \ref{FarnsteinerStrade}  one has $[D_\m:Z_\m]=p^m$, where $D_\m=\mathrm{Frac}(A_\m)$ and $Z_\m=Z(D_\m)$ for some $m\geq 0$. Hence for $q=\mathrm{dim}(H)!$ the assumptions of Theorem \ref{MainResultCEW} are fulfilled and imply that the Hopf algebra action of $H$ on $A$ factors through a group algebra.
\end{proof}

\section{Iterated Differential Operator Rings}

In this last section we show that Cuadra-Etingof-Walton's method can be applied to iterated Ore extensions of derivation type over polynomial rings. The following Proposition is the crucial step to show that the centre of such iterated extensions over a field of characteristic $p$ is large.

\begin{prop}\label{polynomialsubrings}
Let  $A$ be a Noetherian domain over a  field $F$ of characteristic $p$. Suppose that $A$ contains central elements $t_1, \ldots, t_n$, with $n\geq 1$, such that 
\begin{enumerate}
\item $B=F[t_1,\ldots, t_n]$ is a polynomial ring, and 
\item $A$ is a free $B$-module of rank $p^m$, for some $m\geq 0$.
\end{enumerate}
Let $d$ be any $F$-derivation of $A$. Then $A[x;d]$ contains $n+1$ central elements $\tilde{t}_1,\ldots, \tilde{t}_{n+1}$ such that $\widetilde{B}:=F[\tilde{t}_1,\ldots, \tilde{t}_{n+1}]$ is a polynomial ring and $A[x;d]$ is free over $\widetilde{B}$ of rank $p^{m+k}$ for some $k\geq 0$ depending on $d$.
\end{prop}

\begin{proof} 
If $d=0$, then $\widetilde{B}=F[t_1,\ldots, t_n,x]$ is a central subring of $A[x]$ and $A[x]$ has rank $p^m$ over $\widetilde{B}$. Thus assume that $d\neq 0$. For any $a\in A$ we have $d(a^p)=0$ as $\mathrm{char}(F)=p$. Thus $B'=F[t_1^{p}, \ldots, t_n^{p}]$ is a central subring of $A[x;d]$. 

Using the hypothesis (2) we have that $A$ has  rank $p^{n+m}$ over $B'$.
Moreover, the derivation $d$ is a $B'$-linear endomorphism of $A$ and 
by the Cayley-Hamilton Theorem \cite[2.4]{AtiyahMacDonalds}, $d$ will satisfy a monic polynomial $f\in B'[z]$.
Similar to \cite[Lemma 1]{Jacobson} one has that $f$ is a factor of a monic $p$-polynomial  $g\in B'[z]$. 
Recall that a $p$-polynomial is a polynomial whose monomials have $p$ powers as exponents. In order to obtain $g$ one divides $z^{p^i}$ by the monic polynomial $f$, for each $i\geq 0$, to obtain polynomials $q_i, r_i \in B'[z]$ such that
$z^{p^i} = q_i f + r_i$ and $r_i=0$ or $\mathrm{deg}(r_i)<\mathrm{deg}(f)$. Since $B'$ is Noetherian, there must exist $k> 0$ such that  $r_k\in \sum_{i=0}^{k-1} B'r_i$. Thus there are  $a_{k-1}, \ldots, a_1 \in B'$ such that
$g = z^{p^k}+\sum_{i=0}^{k-1}a_{i}z^{p^i} = \left(q_k + \sum_{i=0}^{k-1}a_iq_i\right)f$. 
As $f$ is a factor of $g$ we also have $g(d)\equiv 0$.
Set $\Theta:=x^{p^k} + \sum_{i=0}^{k-1}a_i x^{p^i} \in B'[x]\subset A[x;d]$. 
Note that $\Theta$ commutes with powers of $x$ as the coefficients of $\Theta$  are central in $A[x;d]$.
Furthermore, let $a\in A$, then
$$ \Theta a - a \Theta = d^{p^k}(a) + \sum_{i=0}^{k-1} \lambda_i d^{p^i} (a)  = g(d)(a) = 0.$$
Hence $\Theta$ is central in $A[x;d]$. Since $\Theta$ is monic and of positive degree in $x$ we have that $\Theta$ and $t_1^p, \ldots, t_n^p$ are algebraically independent over $F$. Thus they form a central subring $B'[\Theta]=F[t_1^p, \ldots, t_n^p, \Theta]$ of $A[x;d]$. As $A[x;d]$ has rank $p^k$ over $A[\Theta]$ and as $A[\Theta]$ has rank $p^{n+m}$ over $B'[\Theta]$, we conclude that $A[x;d]$ has rank $p^{n+m+k}$ over $B'[\Theta]$, as desired.\end{proof}

Note that if $B$ is a Noetherian central subring of a Noetherian domain $A$ such that $A$ is finitely generated over $B$, then $D:=\mathrm{Frac}(A)$ can be obtained by inverting the elements of $B$. Hence if $A$ is free of rank $p^n$ over $B$, then $[D:\mathrm{Frac}(B)]=p^n$. In particular, $[D:Z]$ is a power of $p$, where $Z$ denotes the centre of $D$.

Let us call an iterated Ore extension $S[x_1;\alpha_1, d_1][x_2;\alpha_2, d_2][\cdots][x_n; \alpha_n, d_n]$ an iterated Ore extension of derivation type over a commutative domain $S$, if all automorphisms $\alpha_i$ are the identity.
 
\begin{cor}\label{iteratedOre} Any action of a semisimple Hopf algebra $H$ over a field $K$ of characteristic zero on an iterated Ore extension of derivation type over a polynomial ring in finitely many variables factors through a group algebra.
\end{cor}

\begin{proof}
We might assume $S=K$. Suppose that $H$ acts on A and let $R$ be the ring of structure constants of $H$ and $A$. We can consider the $R$-Hopf algebra $H_R$ acting on the $R$-algebra $A_R = R[x_1;d_1][x_2;d_2][\cdots][x_n; d_n]$. For any maximal ideal $\m$ of $R$ we set $F=R/\m$ and $p=\mathrm{char}(F)$. Then ${A_\m} =A_R {\otimes_R} F = F[x_1;d_1][x_2;d_2][\cdots][x_n; d_n]$ is an iterated Ore extension of derivation type over $F$. By Proposition \ref{polynomialsubrings}, ${A_\m}$ contains a central subring $B$, which is a polynomial ring, such $A_\m$ is free over $B$ with  rank a power of $p$. By the previous argument  $D_\m = \mathrm{Frac}(A_\m)$ has a $p$-power as dimension over its centre. Thus, for $q = \mathrm{dim}(H)!$, Theorem $6$ shows that the Hopf algebra action of $H$ on $A$ factors through a group algebra.
\end{proof}

Corollary \ref{iteratedOre} covers the case of the $n$th Weyl algebra over a polynomial ring, i.e. \cite[Proposition 4.3]{CuadraEtingofWalton}, but also other examples like the Jordan plane $A=\CC[x][y; x^2\frac{\partial}{\partial x}]$, which is a partial generalization of  \cite[Theorem 0.1]{ChanWaltonWangZhang} as the assumption that the $H$-action preserves the filtration of $A$ can be removed.

\end{document}